\documentclass[a4paper]{amsart}

\usepackage{amssymb} 

\usepackage[all]{xy}
\usepackage{dsfont}

\SelectTips{cm}{}
\calclayout 

\numberwithin{equation}{section}

\theoremstyle{plain}
\newtheorem{theorem}[equation]{Theorem}
\newtheorem{proposition}[equation]{Proposition}
\newtheorem{lemma}[equation]{Lemma} 
\newtheorem{corollary}[equation]{Corollary}

\theoremstyle{definition}

\theoremstyle{remark} 
\newtheorem{remark}[equation]{Remark} 

\newcommand{\End}{\operatorname{End}}

\newcommand{\Hom}{\operatorname{Hom}}

\newcommand{\Inj}{\operatorname{\mathsf{Inj}}}

\newcommand{\KInj}[1]{\mathsf K(\Inj #1)}

\newcommand{\Loc}{\operatorname{Loc}}

\newcommand{\Spec}{\operatorname{Spec}}
\newcommand{\StMod}{\operatorname{\mathsf{StMod}}}

\newcommand{\supp}{\operatorname{supp}}
\newcommand{\Supp}{\operatorname{Supp}}

\newcommand{\lotimes}{\otimes^{\mathbf L}}

\newcommand{\lto}{\longrightarrow}

\newcommand{\xra}{\xrightarrow}

\def\mcO{\mathcal{O}}

\def\sfT{\mathcal{T}}

\def\sfC{\mathsf C}
\def\sfD{\mathsf D}

\def\sfS{\mathsf S} 
\def\sfT{\mathsf T} 
\def\sfU{\mathsf U}
\def\sfV{\mathsf V}

\def\bbP{\mathbb P}

\newcommand{\fp}{\mathfrak{p}}

\newcommand{\gam}{\varGamma}

\def\one{\mathds 1}

\title[Localising subcategories for cochains on $BG$]{Localising
  subcategories for cochains on the classifying space of a finite
  group}

\author{Dave Benson} 
\address{Dave Benson \\ 
Institute of Mathematics\\ 
University of Aberdeen\\ 
King's College\\ 
Aberdeen AB24 3UE\\ 
Scotland U.K.}

\author{Srikanth B. Iyengar} 
\address{Srikanth B. Iyengar\\ 
Department of Mathematics\\ 
University of Nebraska\\ 
Lincoln, NE 68588\\ 
U.S.A.}

\author{Henning Krause} 
\address{Henning Krause\\ 
Fakult\"at f\"ur Mathematik\\ 
Universit\"at Bielefeld\\ 
33501 Bielefeld\\ 
Germany.}

\begin{document}

\begin{abstract}
  The localising subcategories of the derived category of the cochains
  on the classifying space of a finite group are classified. They are
  in one to one correspondence with the subsets of the set of
  homogeneous prime ideals of the cohomology ring $H^*(G,k)$.
\end{abstract}



\thanks{The research of the second author was undertaken during a
  visit to the University of Bielefeld, supported by a research prize
  from the Humboldt Foundation, and by NSF grant DMS 0903493.}

\maketitle 

\section{Introduction}

Let $G$ be a finite group and $k$ a field of characteristic $p$.  Let $C^*(BG;k)$ be the cochains on the classifying space $BG$. Using the machinery of Elmendorf, K{\v{r}}{\'{\i}}{\v{z}}, Mandell and May \cite{Elmendorf/Kriz/Mandell/May:1996a}, one can regard $C^*(BG;k)$ as a strictly commutative $S$-algebra over the field $k$. The derived category $\sfD(C^*(BG;k))$ has thus a structure of a tensor triangulated category via the left derived tensor product $-\lotimes_{C^*(BG;k)} -$. The unit for the tensor product is $C^*(BG;k)$.

In this paper we apply techniques and results from \cite{Benson/Iyengar/Krause:2008a, Benson/Iyengar/Krause:bik2, Benson/Iyengar/Krause:bik3, Benson/Krause:2008a} to classify the localising subcategories of $\sfD(C^*(BG;k))$. More precisely, there is a notion of stratification for triangulated categories via the action of a graded commutative ring which implies that the localising subcategories are parameterised by sets of homogeneous prime ideals
\cite{Benson/Iyengar/Krause:bik2}.  For $\sfD(C^*(BG;k))$ we use the natural action of the endomorphism ring of the tensor identity which is isomorphic to the cohomology algebra $H^*(G,k)$ of the group $G$.

\begin{theorem}
\label{th:DCBG}
  The derived category $\sfD(C^*(BG;k))$ is stratified by the ring $H^*(G,k)$.
  This yields a one to one correspondence between the localising
  subcategories of $\sfD(C^*(BG;k))$ and subsets of the set of
  homogeneous prime ideals of $H^*(G,k)$.
\end{theorem}



It is proved in \cite{Benson/Krause:2008a} that there is an equivalence of tensor triangulated categories between
$\sfD(C^*(BG;k))$ and the localising subcategory of $\KInj{kG}$ generated by the tensor identity. Here, $\KInj{kG}$ is the homotopy category of complexes of injective ($=$ projective) $kG$-modules, studied in \cite{Benson/Krause:2008a,Krause:2005a}.

The main theorem of \cite{Benson/Iyengar/Krause:bik3} states that
$\KInj{kG}$ is stratified as a tensor triangulated category by
$H^*(G,k)$. Theorem \ref{th:DCBG} is a consequence of a more general
result concerning tensor triangulated categories, which is described
below.

Let $(\sfT,\otimes,\one)$ be a compactly generated tensor triangulated category, as described in \cite[\S8]{Benson/Iyengar/Krause:2008a}, and $R$ a graded commutative noetherian ring acting on $\sfT$ via a homomorphism $R\to \End^{*}_{\sfT}(\one)$. In this case, for each homogeneous prime ideal $\fp$ of $R$ there exists a \emph{local cohomology functor} $\gam_\fp\colon\sfT\to\sfT$; see \cite{Benson/Iyengar/Krause:2008a}. The \emph{support} of an object $X$ in $\sfT$ is then defined to be 
\[
\supp_RX=\{\fp\in\Spec R\mid\gam_\fp X\neq 0\}\,.
\] 
The condition that $\sfT$ is \emph{stratified} by the action of $R$ means that assigning a subcategory $\sfS$ of $\sfT$ to its
support
\[
\supp_R\sfS=\bigcup_{X\in \sfS}\supp_RX
\] 
yields a bijection between \emph{tensor ideal} localising subcategories of $\sfT$ and subsets of the homogeneous prime ideal spectrum $\Spec R$ contained in $\supp_R\sfT$; see \cite[Theorem~4.2]{Benson/Iyengar/Krause:bik2}. Theorem~\ref{th:DCBG} is thus a special case of the result below that relates tensor ideal localising subcategories of $\sfT$ and the localising subcategories of $\Loc_\sfT(\one)$, the localising subcategory of $\sfT$ generated by
the tensor unit. We note that $\Loc_{\sfT}(\one)$ is a compactly generated tensor triangulated category in its own right and that $R$ acts on it as well.

\begin{theorem}
\label{th:T}
Suppose that the Krull dimension of $R$ is finite. If $\sfT$ is
stratified by $R$ as a tensor triangulated category, then so is
$\Loc_\sfT(\one)$, and there is a bijection
\[
\left\{
\begin{gathered}
\text{Tensor ideal localising}\\ \text{subcategories of $\sfT$}\end{gathered}\;
\right\}
\stackrel{\sim}\lto
\left\{
\begin{gathered}
\text{Localising subcategories}\\ \text{of $\Loc_{\sfT}(\one)$}
\end{gathered}\; \right\}.
\] 
It assigns each tensor ideal localising subcategory $\sfS$ of $\sfT$
to $\sfS\cap\Loc_{\sfT}(\one)$.
\end{theorem}

\begin{remark}\label{rk:P1}
  The theorem is not true without the assumption that $\sfT$ is
  stratified by $R$. For example, let $\sfT$ be the derived category
  of quasi-coherent sheaves on the projective line $\bbP^1_k$. The
  tensor unit is $\mcO$.  In this example there are no proper
  localising subcategories of $\Loc_\sfT(\mcO)$ since
  $\End_\sfT^*(\mcO)=k$, while there are many tensor ideal localising
  subcategories of $\sfT$.
\end{remark}

\begin{remark}
  The assumption that the Krull dimension of $R$ is finite is
  artificial, and is used only to ensure that for each $X\in\sfT$ and
  $\fp\in\Spec R$ the object $\gam_{\fp}X$ belongs to
  $\Loc_\sfT(X)$. One can replace this condition by, for instance, the
  assumption that $\sfT$ arises as the homotopy category of a Quillen
  model category \cite[\S6]{Stevenson:2011a}.
\end{remark}

\section{Localising subcategories of $\Loc_{\sfT}(\one)$}

In this section $\sfT$ is a triangulated category with set-indexed coproducts and the tensor product $\otimes$ provides a symmetric monoidal structure with unit $\one$ on $\sfT$, which is exact in each variable and preserves set-indexed coproducts.

The proof of Theorem~\ref{th:T} is based on a sequence of elementary lemmas. The first one describes the tensor ideal localising subcategory of $\sfT$ which is generated by a class $\sfC$ of objects; we denote this by $\Loc_\sfT^\otimes(\sfC)$.

\begin{lemma}
\label{le:gen}
Let $\sfC$ be a class of objects of $\sfT$. Then
\[
\Loc_\sfT^\otimes(\sfC)=\Loc_\sfT(\{X\otimes Y\mid X\in\sfC,
Y\in\sfT\}).
\]
\end{lemma}

\begin{proof}
  Set $\sfS=\Loc_\sfT(\{X\otimes Y\mid X\in\sfC, Y\in\sfT\})$.  It
  suffices to show that $\sfS$ is tensor ideal. This means that
  $F\sfS\subseteq \sfS$ for each tensor functor $F=-\otimes Y$, which
  is an immediate consequence of Lemma~\ref{le:gen2} below.
\end{proof}

\begin{lemma}
\label{le:gen2}
Let $F\colon \sfU\to\sfV$ be an exact functor between triangulated
categories that preserves set-indexed coproducts.  If $\sfC$ is a
class of objects of $\sfU$, then
\[
F\Loc_\sfU(\sfC)\subseteq\Loc_\sfV(F\sfC).
\]
\end{lemma}

\begin{proof}
The preimage $F^{-1}\Loc_\sfV(F\sfC)$ is a localising subcategory of $\sfU$ containing $\sfC$.  Thus it contains $\Loc_\sfU(\sfC)$, and one gets 
\[
F\Loc_\sfU(\sfC)\subseteq F F^{-1}\Loc_\sfV(F\sfC)\subseteq\Loc_\sfV(F\sfC).\qedhere
\]
\end{proof}

\begin{lemma}\label{le:loc}
  Let $\gam\colon\sfT\to\sfT$ be a colocalization functor that
  preserves set-indexed coproducts. Then for any $X\in\sfT$ and $Y\in
  \Loc_\sfT(\one)$, there is a natural isomorphism
  \[\gam X\otimes Y\xra{\sim} \gam(X \otimes Y). 
\]
\end{lemma}

\begin{remark}
  There is an analogous result for a localization functor
  $L\colon\sfT\to\sfT$ that preserves set-indexed coproducts: For any
  $X\in\sfT$ and $Y\in \Loc_\sfT(\one)$, there is a natural
  isomorphism $L(X \otimes Y) \xra{\sim} LX \otimes Y$.
\end{remark}

\begin{proof}
  A colocalisation functor $\gam$ comes with a natural morphism $\gam
  X\to X$.  Tensoring this with an object $Y\in \Loc_\sfT(\one)$ gives
  a morphism $\gam X\otimes Y\to X\otimes Y$ that factors through the
  natural morphism $\gam (X\otimes Y)\to X \otimes Y$. Here, one uses
  that $\gam X\otimes Y$ belongs to $\gam\sfT$, since the
  objects $Y'\in\sfT$ with $\gam X\otimes Y'\in\gam\sfT$ form a
  localising subcategory containing $\one$. The induced morphism
  $\phi_Y\colon\gam X\otimes Y\to \gam(X \otimes Y)$ is an
  isomorphism.  To see this, observe that the objects $Y'\in\sfT$ such
  that $\phi_{Y'}$ is an isomorphism form a localising subcategory
  containing $\one$.
\end{proof}

\begin{proposition}
\label{pr:FLocS}
Suppose that the unit $\one$ is compact in $\sfT$ and let
$\gam\colon\sfT\to\Loc_\sfT(\one)$ denote the right adjoint of the
inclusion $\Loc_\sfT(\one)\to\sfT$. If $\sfS$ is a localising
subcategory of $\Loc_\sfT(\one)$, then
\[
\Loc^\otimes_\sfT(\sfS)\cap\Loc_\sfT(\one)=\gam(\Loc^\otimes_\sfT(\sfS))=\sfS. 
\]
\end{proposition}

\begin{proof}
We verify each of the following inclusions
\[\sfS\subseteq\Loc^\otimes_\sfT(\sfS)\cap\Loc_\sfT(\one)\subseteq
\gam(\Loc^\otimes_\sfT(\sfS))\subseteq\sfS. \] The first one is clear.
Composing the functor $\gam$ with the inclusion
$\Loc_\sfT(\one)\to\sfT$ yields a colocalisation functor that
preserves set-indexed coproducts, since $\one$ is compact.  For an
object $X$ in $\Loc^\otimes_\sfT(\sfS)\cap\Loc_\sfT(\one)$, we have
$\gam X\cong X$. This gives the second inclusion.  Applying
Lemma~\ref{le:loc} together with the description of
$\Loc^\otimes_\sfT(\sfS)$ from Lemma~\ref{le:gen} yields the third
inclusion. 
\end{proof}

\begin{corollary}
  Suppose that the unit $\one$ is a compact object in $\sfT$. Assigning each
  localising subcategory $\sfS$ of $\Loc_\sfT(\one)$ to
  $\Loc^\otimes_{\sfT}(\sfS)$ gives a bijection
\[
\left\{
\begin{gathered}
\text{Localising subcategories}\\ \text{of $\Loc_{\sfT}(\one)$}
\end{gathered}\; \right\} 
\stackrel{\sim}\lto
\left\{
\begin{gathered}
\text{Tensor ideal localising subcategories of}\\ \text{$\sfT$
  generated by objects from $\Loc_\sfT(\one)$}\end{gathered}\;
\right\}.
\] 
\end{corollary}

\begin{proof}
The inverse map sends $\sfU\subseteq\sfT$ to $\sfU\cap\Loc_\sfT(\one)$.
\end{proof}

We are now ready to prove Theorem~\ref{th:T}. Note that in this $\sfT$
is a compactly generated tensor triangulated category, which entails a
host of additional requirements; see
\cite[\S8]{Benson/Iyengar/Krause:2008a} for a list.

\begin{proof}[Proof of Theorem \ref{th:T}]
It follows from Proposition~\ref{pr:FLocS} that the assignment
\[ 
  \sfS\longmapsto \Loc^\otimes_\sfT(\sfS) 
\] 
is an injective map from the localising subcategories of
$\Loc_\sfT(\one)$ to the tensor ideal localising subcategories of
$\sfT$. In general, it is not bijective, as the example of
Remark~\ref{rk:P1} shows. However, since $\sfT$ is stratified by $R$
as a tensor triangulated category, it follows from
\cite[\S7]{Benson/Iyengar/Krause:bik2} that each tensor ideal
localising subcategory is generated by a set of objects of the form
$\gam_\fp\one$.  Since $R$ has finite Krull dimension,
\cite[Theorem~3.4]{Benson/Iyengar/Krause:bik2} yields that
$\gam_\fp\one$ is in $\Loc_\sfT(\one)$.  Therefore, given a tensor
ideal localising subcategory $\sfU$ of $\sfT$, the localising
subcategory
\[ 
\sfU'=\Loc_\sfT(\{\gam_{\fp} \one \mid \fp\in\Supp_{R}\sfU\})\subseteq
\Loc_{\sfT}(\one)
\] 
satisfies $\Loc^\otimes_\sfT(\sfU')=\sfU$.  This proves the
surjectivity of the assignment. Moreover, we have shown that each
localising subcategory of $\Loc_\sfT(\one)$ is generated by objects of
the form $\gam_\fp\one$, so $\Loc_\sfT(\one)$ is stratified by the
action of $R$; see \cite[Theorem~4.2]{Benson/Iyengar/Krause:bik2}.
\end{proof}

\section{The cohomological nucleus}

Let $(\sfT,\otimes,\one)$ be a compactly generated tensor triangulated
category and let $R$ be a graded commutative noetherian ring acting on
$\sfT$ via a homomorphism $R\to \End^{*}_{\sfT}(\one)$. Suppose in
addition that $R$ has finite Krull dimension.

We define the \emph{cohomological nucleus} of $\sfT$ as the set of
homogeneous prime ideals $\fp$ of $R$ such that there exists an object
$X\in\sfT$ satisfying $\Hom_\sfT^*(\one, X)=0$ and $\gam_\fp X\neq
0$. This definition is motivated by work of Benson, Carlson, and
Robinson in the context of modular group representations
\cite{Benson/Carlson/Robinson:1990a}.

For $\fp$ in $\Spec R$ consider the tensor ideal localising subcategory
\[
\gam_\fp\sfT=\{Y\in\sfT\mid Y\cong \gam_\fp X\text{ for some }X\in\sfT\}.
\]
Note that an object $X\in\sfT$ belongs to $\gam_\fp\sfT$ if and only
if $\Hom_\sfT^*(C,X)$ is $\fp$-local and $\fp$-torsion for every
compact $C\in\sfT$, by
\cite[Corollary~4.10]{Benson/Iyengar/Krause:2008a}. The result below
gives a local description of the cohomological nucleus.

\begin{proposition}
Let $\fp$ be a homogeneous prime ideal of $R$. The following conditions are equivalent:
\begin{enumerate}
\item Every object $X$ in $\sfT$ with $\Hom_\sfT^*(\one, X)=0$ satisfies
  $\gam_\fp X= 0$.
\item One has $\Loc_{\sfT}(\gam_{\fp}\one)=\gam_\fp\sfT$.
\item Every localising subcategory of $\gam_\fp\sfT$ is a tensor ideal
  of $\sfT$.
\end{enumerate}
\end{proposition}

\begin{proof}
  The Krull dimension of $R$ is finite, so $\gam_\fp X$ is in
  $\Loc_\sfT(X)$ for each $X$ in $\sfT$, by
  \cite[Theorem~3.4]{Benson/Iyengar/Krause:bik2}. This fact is used
  without further comment.

  (1) $\Rightarrow$ (2): Set $\sfS=\Loc_{\sfT}(\gam_{\fp}\one)$. Note
  that $\sfS\subseteq \gam_\fp\sfT$; we claim that equality
  holds. Indeed, $\sfS\subseteq\Loc_{\sfT}(\one)$ and also
  $\Loc_{\sfT}^{\otimes}(\sfS)=\gam_{\fp}\sfT$, since
  $\gam_\fp=\gam_\fp\one\otimes -$. Thus, for any $X$ in
  $\gam_\fp\sfT$ from Proposition~\ref{pr:FLocS} one gets an exact
  triangle $\gam X\to X\to X'\to$ with $\gam X\in\sfS$ and
  $\Hom_\sfT^*(\one,X)=0$. Then (1) implies $X'=0$ and hence
  $X\in\sfS$.

  (2) $\Rightarrow$ (3): Let $\sfS$ be a localising subcategory of
  $\gam_\fp\sfT$. Using (2) and the fact that $\gam_\fp\sfT$ is a
  tensor ideal of $\sfT$, one has
  $\Loc_\sfT^\otimes(\sfS)\subseteq\Loc_\sfT(\one)$. Then it follows,
  again from Proposition~\ref{pr:FLocS}, that $\sfS$ is a tensor ideal
  of $\sfT$.

  (3) $\Rightarrow$ (1): Assume $\Hom_\sfT^*(\one,X)=0$; then
  $\Hom_\sfT^*(\one,\gam_{\fp}X)=0$, as $\one$ is compact. Condition
  (3) implies that $\Loc_\sfT(\gam_\fp\one)=\gam_\fp\sfT$. Thus
  $\gam_\fp X$ belongs to $\Loc_\sfT(\gam_\fp\one)$ and therefore also
  to $\Loc_\sfT(\one)$. So one obtains
  $\Hom_\sfT^*(\gam_{\fp}X,\gam_{\fp}X)=0$, which implies
  $\gam_{\fp}X=0$.
 \end{proof}

 Consider as an example for $\sfT$ the stable module category $\StMod
 kG$ of a finite group $G$ with the canonical action of $R=H^*(G,k)$.
 We refer to \cite{Benson:1994a,Benson/Carlson/Robinson:1990a} for the
 discussion of two variations of the nucleus, namely the \emph{group
   theoretic} and the \emph{representation theoretic} nucleus. There
 it is shown that $\Loc_\sfT(\one)=\sfT$ if and only if the
 centraliser of every element of order $p$ in $G$ is $p$-nilpotent and
 every block is either principal or semisimple, where $p$ denotes the
 characteristic of the field $k$.

It is convenient to define for any class $\sfC$ of objects of $\sfT$
\begin{align*}
\sfC^\perp&=\{Y\in\sfT\mid\Hom^*_\sfT(X,Y)=0\text{ for all }X\in\sfC\},\\
^\perp\sfC&=\{X\in\sfT\mid\Hom^*_\sfT(X,Y)=0\text{ for all }Y\in\sfC\}.
\end{align*}

Now let $\sfS=\Loc_\sfT(\one)$.  The \emph{representation theoretic
  nucleus} is by definition
\[
\bigcup_{X\in \sfS^\perp\cap\sfT^c}\supp_R X.
\] 
Clearly, this is contained in the cohomological nucleus.  It is a
remarkable fact that the representation theoretic nucleus is non-empty
if $\sfS^\perp\neq 0$; this is proved in
\cite{Benson:1994a,Benson/Carlson/Robinson:1990a}.  Moreover,
Question~13 of \cite{Carlson:2000a} asks whether
$\sfS={^\perp(\sfS^\perp\cap\sfT^c)}$.  Note that
$\sfS={^\perp(\sfS^\perp)}$ follows from general principles.

\subsection*{Acknowledgments} 
It is a pleasure to thank Greg Stevenson for helpful comments on this work.

\end{document}